\documentclass[11pt]{amsart}

\usepackage[
  textwidth=5.4in,
  textheight=8.2in,
  centering
]{geometry}

\usepackage{amsmath,amssymb,amsfonts}
\usepackage{amscd}
\usepackage{hyperref}

\newtheorem{theorem}{Theorem}
\newtheorem{lemma}{Lemma}

\newtheorem{corollary}{Corollary}

\newcommand{\Z}{\mathbb Z}
\newcommand{\Ztwo}{\mathbb Z_2}
\newcommand{\Sphere}{\mathbb S}
\newcommand{\coind}{\operatorname{coind}}
\newcommand{\ind}{\operatorname{ind}}
\newcommand{\B}{\mathrm B}

\newcommand{\susp}{\operatorname{susp}}

\title{On the Gap Between the Co-Indices of a Free \texorpdfstring{$\Ztwo$}{Z2}-Space and Its Suspension}

\author{Hamid Reza Daneshpajouh}
\address{School of Mathematical Sciences, University of Nottingham Ningbo China, 199 Taikang East Road, Ningbo, 315100, China}
\email{Hamid-Reza.Daneshpajouh@nottingham.edu.cn}

\begin{document}

\maketitle

\begin{abstract}
For a free \(\Ztwo\)-space \(X\), the co-index \(\coind(X)\) is the largest
integer \(m\) for which there exists a \(\Ztwo\)-equivariant map
\(\Sphere^m\to X\), where \(\Sphere^m\) carries the antipodal action. Since
suspension sends such a map to a \(\Ztwo\)-equivariant map
\(
\Sphere^{m+1}\rightarrow \susp(X),
\)
one always has
\[
\coind(\susp(X))\geq \coind(X)+1.
\]
We prove that the excess over this lower bound can be arbitrarily large. More
precisely, for every \(n\geq 2\), we construct a finite free \(n\)-dimensional
simplicial \(\Ztwo\)-complex \(\mathcal K\) such that \(\coind(\mathcal K)=1\), and \(\coind(\susp(\mathcal K))=n+1\). This answers a question of Simonyi, Tardos, and Vre\'cica on the possible
growth of co-index under suspension and, equivalently, shows that the
co-index lower bound on the chromatic number of a graph \(G\) obtained from
\(\B_0(G)\) can exceed the corresponding bound obtained from the box complex
\(\B(G)\) by an arbitrarily large amount.
\end{abstract}

\section{Introduction}
Throughout the paper, we use the following convention for suspension. Let \(X\)
be a \(\Ztwo\)-space, and denote its involution by \(x\mapsto -x\). We define
\[
\susp(X):=(X\times[-1,1])/\sim,
\]
where all points of \(X\times\{1\}\) are identified to one suspension point and
all points of \(X\times\{-1\}\) are identified to the other suspension point.
The induced \(\Ztwo\)-action on \(\susp(X)\) is given by \(
-[x,t]=[-x,-t]\). In particular, if the
\(\Ztwo\)-action on \(X\) is free, then the induced \(\Ztwo\)-action on
\(\susp(X)\) is also free. With this convention,
\(\susp(\Sphere^m)\) is naturally \(\Ztwo\)-homeomorphic to
\(\Sphere^{m+1}\) with the antipodal action. Recall that, for a free \(\Ztwo\)-space \(X\), the co-index
\(\coind(X)\) is the largest integer \(m\) for which there exists a
\(\Ztwo\)-equivariant map
\(
\Sphere^m\rightarrow X,
\)
where \(\Sphere^m\) carries the antipodal action. If such a map exists, then
suspending it gives a \(\Ztwo\)-equivariant map \(
\Sphere^{m+1}\rightarrow \susp(X)
\). Hence
\begin{equation}
\coind(\susp(X))\geq \coind(X)+1.
\label{eq:co-susp}
\end{equation}

The aim of this paper is to show that the excess in
\eqref{eq:co-susp} can be made arbitrarily large for suitable finite free
\(\Ztwo\)-complexes \(X\). This gives an affirmative answer to a question of
Simonyi, Tardos, and Vre\'cica~\cite{simonyi2009local} concerning the possible
growth of co-index under suspension. Although this question is of independent
interest in equivariant topology, it is also strongly motivated by topological
methods in graph coloring.

These methods originate with Lov\'asz's proof of Kneser's conjecture
\cite{lovasz1978kneser}. A central idea is to associate to a graph \(G\) a free
\(\Ztwo\)-space whose equivariant topology yields lower bounds on the
chromatic number \(\chi(G)\) via the Borsuk--Ulam theorem. Among the most
important such spaces are the box complexes \(\B(G)\) and \(\B_0(G)\).

Briefly, \(\B(G)\) is the free \(\Ztwo\)-complex whose simplices encode ordered
pairs of disjoint vertex sets forming complete bipartite subgraphs of \(G\),
with the additional common-neighbor condition in the cases where one of the
parts is empty. The \(\Ztwo\)-action interchanges the two parts. The variant
\(\B_0(G)\) is obtained by omitting the common-neighbor condition. For precise
definitions and further background, we refer the reader to
\cite{Daneshpajouh2025}. These complexes give the standard co-index bounds
\begin{equation}
\chi(G)\geq \coind(\B(G))+2,
\label{eq:box}
\end{equation}
and
\begin{equation}
\chi(G)\geq \coind(\B_0(G))+1.
\label{eq:box0}
\end{equation}

The relation between these two complexes is governed by a theorem of
Csorba~\cite{csorba2007homotopy}, which states that \(\B_0(G)\) is
\(\Ztwo\)-homotopy equivalent to the suspension of \(\B(G)\):
\begin{equation}
\susp(\B(G))\simeq_{\Ztwo}\B_0(G).
\label{eq:suspension}
\end{equation}
Applying \eqref{eq:co-susp} to \(X=\B(G)\) and using
\eqref{eq:suspension}, we obtain
\[
\coind(\B_0(G))\geq \coind(\B(G))+1.
\]
Consequently, the bound \eqref{eq:box0} is always at least as strong as the
bound \eqref{eq:box}.

It is therefore natural to ask to what extent the \(\B_0(G)\)-bound can improve
upon the \(\B(G)\)-bound; equivalently, how large the difference
\[
\coind(\B_0(G))-\bigl(\coind(\B(G))+1\bigr)
\]
can be. By Csorba's universality theorem for box complexes
\cite[Theorem~5.1]{csorba2007homotopy}, every finite free simplicial
\(\Ztwo\)-complex arises, up to \(\Ztwo\)-homotopy equivalence, as \(\B(G)\)
for some graph \(G\). Hence, using the suspension equivalence
\eqref{eq:suspension}, the graph-theoretic problem is equivalent to the
topological suspension problem above when restricted to finite free simplicial
\(\Ztwo\)-complexes. We shall work in this finite simplicial setting
throughout.

Simonyi, Tardos, and Vre\'cica~\cite[Section~5]{simonyi2009local} established
that this gap can attain the value \(1\). More precisely, they gave explicit
constructions based on compact orientable two-dimensional manifolds \(T\) of
positive even genus for which
\[
\coind(T)=1,
\qquad
\coind(\susp(T))=3.
\]
Subsequently, examples derived from Brieskorn manifolds yielded spaces \(X\)
satisfying
\[
\coind(X)=1,
\qquad
\coind(\susp(X))=4;
\]
see~\cite[Proposition~8.1]{Daneshpajouh2025}. Thus the gap can also attain the
value \(2\). The general question of whether the gap is unbounded remained open. The main result of this paper gives an affirmative answer.

\begin{theorem}\label{thm:main}
For every integer \(n\geq 2\), there exists a finite \(n\)-dimensional
simplicial complex \(\mathcal K\) equipped with a free \(\Ztwo\)-action such
that
\[
\coind(\mathcal K)=1
\qquad\text{and}\qquad
\coind(\susp(\mathcal K))=n+1.
\]
\end{theorem}

Combining Theorem~\ref{thm:main} with~\ref{eq:suspension}, and Csorba's universality theorem gives the
following graph-theoretic form of the result.
\begin{corollary}\label{cor:graph-gap}
For every integer \(n\geq 2\), there exists a graph \(G_n\) such that
\[
\coind(\B(G_n))=1
\qquad\text{and}\qquad
\coind(\B_0(G_n))=n+1.
\]
\end{corollary}

Thus the chromatic lower bound obtained from the \(\B_0(G)\)-bound, namely
\eqref{eq:box0}, can be arbitrarily stronger than the \(\B(G)\)-bound, namely \eqref{eq:box}. Together with other useful properties of \(\B_0(G)\), such as the decidability of its connectivity
\cite[Theorem~5.1]{Daneshpajouh2025} and its natural behavior under joins
\cite[Proposition~7.3]{Daneshpajouh2025}, this suggests that \(\B_0(G)\) can be
a more effective practical invariant than \(\B(G)\) in certain situations.

As another application, we obtain a new family of non-tidy spaces. Recall that for
a free \(\Ztwo\)-space \(X\), the index \(\ind(X)\) is the least integer \(r\)
for which there exists a \(\Ztwo\)-equivariant map
\(
X\rightarrow \Sphere^r
\). By the Borsuk--Ulam theorem,
\begin{equation}
\coind(X)\leq \ind(X).
\label{eq:coind-ind}
\end{equation}
Following Matou\v{s}ek~\cite{matousek2008using}, a free \(\Ztwo\)-space is
called non-tidy if this inequality is strict. Examples of non-tidy spaces were
previously known; see, for instance, \cite[pp.~100--101]{matousek2008using}, \cite{matsushita2017examples}
and~\cite{daneshpajouh2023hedetniemi}.

Theorem~\ref{thm:main} gives another source of finite non-tidy complexes with
arbitrarily large gap between index and co-index. Let \(\mathcal K\) be the
\(n\)-dimensional complex from Theorem~\ref{thm:main}. Then,  \(\ind(\susp(\mathcal K))\geq n+1\) by \eqref{eq:coind-ind}. Since the index of a finite free \(\Ztwo\)-simplicial complex is always bounded above by its dimension~\cite[Proposition 5.3.2]{matousek2008using}, it follows that
\(
\ind(\susp(\mathcal{K})) = n+1
\). Moreover, suspending an equivariant map \(X\to \Sphere^r\) gives an equivariant
map
\(
\susp(X)\rightarrow \Sphere^{r+1},
\)
and therefore
\begin{equation}
\ind(\susp(X))\leq \ind(X)+1.
\label{eq:ind-suspension}
\end{equation}
Hence
\(
n+1=\ind(\susp(\mathcal K))\leq \ind(\mathcal K)+1,
\)
so \(\ind(\mathcal K)\geq n\). Since \(\dim\mathcal K=n\), again by the
dimension upper bound, \(\ind(\mathcal K)\leq n\). Therefore
\(\ind(\mathcal K)=n\), and we obtain the following corollary.

\begin{corollary}\label{cor:nontidy}
For every integer \(n\geq 2\), there exists a finite \(n\)-dimensional
simplicial complex \(\mathcal K\) with a free \(\Ztwo\)-action such that
\[
\coind(\mathcal K)=1
\qquad\text{and}\qquad
\ind(\mathcal K)=n.
\]
In particular, finite non-tidy free \(\Ztwo\)-complexes exist with arbitrarily
large gap between index and co-index.
\end{corollary}

The examples in Theorem~\ref{thm:main} all have co-index \(1\). We end the introduction with a simple modification that gives the following
more flexible realization statement.

\begin{corollary}\label{cor:flexible}
For any integers \(m\geq 1,n\geq 2\) with \(1\leq m\leq n-1\), there exists a finite
simplicial complex \(\mathcal K^*\) with a free \(\Ztwo\)-action such that
\[
\coind(\mathcal K^*)=m
\qquad\text{and}\qquad
\coind(\susp(\mathcal K^*))=n+1.
\]
\end{corollary}

\begin{proof}
Let \(\mathcal K\) be the complex given by Theorem~\ref{thm:main} in dimension
\(n\). Thus
\(\coind(\mathcal K)=1\), and \(\coind(\susp(\mathcal K))=n+1\). Define
\(
\mathcal K^* = \mathcal K\,\sqcup\, \Sphere^m
\), where \(\Sphere^m\) carries the antipodal action, and equip \(\mathcal K^*\)
with the induced free \(\Ztwo\)-action on each component.

Since the co-index of a disjoint union is the maximum of the co-indices of its
components, we have
\[
\coind(\mathcal K^*)
=
\max\{\coind(\mathcal K),\coind(\Sphere^m)\}
=
\max\{1,m\}
=
m.
\]

It remains to compute the co-index of \(\susp(\mathcal K^*)\). Since
\(\mathcal K\subseteq \mathcal K^*\), there is an equivariant inclusion
\(
\susp(\mathcal K)\rightarrow \susp(\mathcal K^*)
\). Therefore
\[
\coind(\susp(\mathcal K^*))\geq \coind(\susp(\mathcal K))=n+1.
\]
On the other hand, since \(m\leq n-1\), we have
\[
\dim \mathcal K^*=\max\{\dim \mathcal K,m\}=n.
\]
Hence \(\dim \susp(\mathcal K^*)=n+1\). Since the co-index of a finite free \(\Ztwo\)-complex is bounded above by its
dimension~\cite[Proposition~5.3.2]{matousek2008using}, we get
\(
\coind(\susp(\mathcal K^*))\leq n+1
\). Combining the two inequalities gives \(
\coind(\susp(\mathcal K^*))=n+1
\).
\end{proof}

\section{Proof of the Main Result}

The proof of Theorem~\ref{thm:main} proceeds in two steps. We first establish a
topological criterion ensuring that a finite free \(\Ztwo\)-complex
\(\mathcal K\) has
\(
\coind(\mathcal K)=1\), and \(\coind(\susp(\mathcal K))=\dim \mathcal K+1\).
This criterion is similar to a proposition used by the author and Meunier
in~\cite[Proposition 8.1]{Daneshpajouh2025}; it provides a mechanism for making the co-index of
the suspension as large as possible while the original complex still has
co-index one. We then prove that such complexes exist by applying Leary's metric Kan--Thurston theorem~\cite{Leary} to 
spheres. 

\begin{lemma}\label{lem:criterion}
Let \(\mathcal K\) be a finite connected free simplicial \(\Ztwo\)-complex of
dimension \(n\). Assume that:
\begin{enumerate}
    \item \(\mathcal K\) is aspherical, that is,
    \(
    \pi_i(\mathcal K)=0 \qquad \text{for all } i\geq 2;
    \)
    \item
    \(
    \widetilde H_i(\mathcal K;\Z)=0 \qquad \text{for all } i<n.
    \)
\end{enumerate}
Then
\(
\coind(\mathcal K)=1\), and \(\coind(\susp(\mathcal K))=n+1\).
\end{lemma}

\begin{proof}
Since \(\widetilde H_0(\mathcal K;\Z)=0\), the complex \(\mathcal K\) is path
connected. Choose a point \(b\in \mathcal K\) and a path \(\gamma\) from \(b\)
to its antipode \(-b\). Mapping one semicircle of \(\Sphere^1\) to \(\gamma\)
and the opposite semicircle to its \(\Ztwo\)-translate gives a
\(\Ztwo\)-equivariant map \(\Sphere^1\rightarrow \mathcal K\). Hence \(\coind(\mathcal K)\geq 1\).

We prove that \(\coind(\mathcal K)\leq 1\). Suppose, to the contrary, that
there exists a \(\Ztwo\)-equivariant map \(f_2\colon \Sphere^2\rightarrow \mathcal K\). Since \(\pi_2(\mathcal K)=0\), \(f_2\) extends over a \(3\)-ball. Viewing
\(\Sphere^3\) as the union of two closed hemispheres exchanged by the antipodal
involution, we extend \(f_2\) over one hemisphere and define the map on the
other hemisphere by equivariance. This gives a \(\Ztwo\)-equivariant map \(
f_3\colon \Sphere^3\rightarrow \mathcal K\). Repeating the same argument, using \(\pi_i(\mathcal K)=0\) for all
\(i\geq 2\), we obtain \(\Ztwo\)-equivariant maps
\[
\Sphere^\ell\rightarrow \mathcal K
\qquad\text{for every } \ell\geq 2.
\]

This is impossible. Indeed, since \(\mathcal K\) is a finite free
\(\Ztwo\)-complex of dimension \(n\), there exists a \(\Ztwo\)-equivariant map
\(
\mathcal K\rightarrow \Sphere^n
\)
by~\cite[Proposition~5.3.2]{matousek2008using}. Composing such a map with the
equivariant map \(\Sphere^{n+1}\to \mathcal K\) obtained above gives a
\(\Ztwo\)-equivariant map \(\Sphere^{n+1}\rightarrow \Sphere^n,\)
contradicting the Borsuk--Ulam theorem. Hence \(\coind(\mathcal K)=1\).

We now consider the suspension \(\susp(\mathcal K)\). Since reduced homology
shifts under suspension,
\[
\widetilde H_j(\susp(\mathcal K);\Z)
\cong
\widetilde H_{j-1}(\mathcal K;\Z)
\qquad\text{for all } j.
\]
By assumption,
\[
\widetilde H_j(\susp(\mathcal K);\Z)=0
\qquad\text{for } 1\leq j\leq n.
\]
Moreover, since \(\mathcal K\) is connected, \(\susp(\mathcal K)\) is simply
connected. The Hurewicz theorem therefore implies that
\[
\pi_j(\susp(\mathcal K))=0
\qquad\text{for } 1\leq j\leq n.
\]

Starting from the tautological equivariant map \(
\Sphere^0\rightarrow \susp(\mathcal K),\)
which sends the two points of \(\Sphere^0\) to the two suspension points, the
same extension-over-hemispheres argument gives, inductively, a
\(\Ztwo\)-equivariant map \(
\Sphere^{n+1}\rightarrow \susp(\mathcal K)\).
Thus
\[
\coind(\susp(\mathcal K))\geq n+1.
\] 
On the other hand, \(\susp(\mathcal K)\) is a finite free simplicial
\(\Ztwo\)-complex of dimension \(n+1\), and therefore \(\coind(\susp(\mathcal K))\leq\dim(\susp(\mathcal K))= n+1\) by~\cite[Proposition 5.3.2]{matousek2008using}.
Consequently,
\(
\coind(\susp(\mathcal K))=n+1
\).
\end{proof}

We now prove the existence of complexes satisfying the hypotheses of
Lemma~\ref{lem:criterion}. The construction uses Leary's metric
Kan--Thurston theorem~\cite{Leary}. We first recall the part of the theorem that is relevant to our present discussion.

\begin{theorem}[Leary's metric Kan--Thurston theorem {\cite[Theorem~A]{Leary}}]
\label{thm:leary}
Let \(X\) be a simplicial complex. There is a locally CAT(0) cubical complex
\(T_X\) and a map
\[
t_X:T_X\rightarrow X
\]
with the following properties.
\begin{enumerate}
    \item The map \(t_X\) induces an isomorphism on homology for any local
    coefficients on \(X\).
    \item If \(X\) is finite, then \(T_X\) is finite. The dimension of \(T_X\) is equal
    to the dimension of \(X\), except when \(X\) is two-dimensional, in which case
    \(T_X\) is three-dimensional.
\end{enumerate}
Moreover, the construction is functorial in the following sense:
\begin{enumerate}
    \item[(i)] If
\(
f:X\rightarrow Y
\)
is a simplicial map which is injective on each simplex, then there is an
induced cubical map
\(
T(f):T_X\rightarrow T_Y.
\)
\item[(ii)] 
If \(f\) is injective, then \(T(f)\) embeds \(T_X\) isometrically as a totally
geodesic subcomplex of \(T_Y\). If \(f\) is locally injective, then \(T(f)\) is
locally isometric.
\end{enumerate}
If \(X\) is the union of subcomplexes \(X_\alpha\), then \(T_X\) is equal to the
union of the corresponding subcomplexes \(T_{X_\alpha}\). 
\end{theorem}

We shall apply Theorem~\ref{thm:leary} to an antipodally triangulated sphere.
Let \(X=\Sphere^n\) be the boundary complex of the \((n+1)\)-dimensional
crosspolytope, and let
\(
a:X\rightarrow X
\)
be the antipodal map. Then \(a\) is a fixed-point-free simplicial involution.
Since \(a\) is a simplicial automorphism, Leary's functoriality gives an
induced cubical involution
\[
T(a):T_X\rightarrow T_X.
\]
The complex \(T_X\) will provide the required homological and homotopical
properties. The only remaining issue is to show that the induced involution
\(T(a)\) is fixed-point-free. This follows from the naturality of the maps
\(t_X:T_X\to X\) with respect to the induced maps \(T(f)\). Although this
naturality is not stated separately in Leary's Theorem~A, it follows from the
functorial construction in \cite[Theorem~5.1]{Leary}. We record the argument
below.

\begin{lemma}\label{lem:leary-naturality}
Let \(f:X\to Y\) be a simplicial map which is injective on each simplex. Then
the induced cubical map
\[
T(f):T_X\rightarrow T_Y
\]
is compatible with the maps \(t_X\) and \(t_Y\), in the sense that
\[
t_Y\circ T(f)=f\circ t_X.
\]
In other words, the following diagram commutes:
\[
\begin{CD}
T_X @>{T(f)}>> T_Y\\
@V{t_X}VV       @VV{t_Y}V\\
X @>{f}>> Y .
\end{CD}
\]
Moreover, the assignment \(X\mapsto T_X\), \(f\mapsto T(f)\), is functorial on
the category of simplicial maps which are injective on each simplex.
\end{lemma}

\begin{proof}
The functoriality is part of Leary's construction in
\cite[Theorem~5.1]{Leary}. We recall why the compatibility with the maps
\(t_X\) also follows from the construction.

Leary constructs \(T_X\) functorially by assembling it over the simplices of
\(X\). If \(X\) is obtained from a subcomplex \(W\) by attaching a simplex
\(\sigma\), then \(T_X\) is obtained from \(T_W\) by attaching the corresponding
mapping-cylinder piece. The map \(t_X:T_X\to X\) agrees with \(t_W\) on
\(T_W\), sends the terminal part of the new piece to the barycentre of
\(\sigma\), and is defined on the remaining part by linear interpolation.

Now let \(f:X\to Y\) be injective on each simplex. Then \(f\) sends each
simplex \(\sigma\subset X\) isomorphically onto the simplex \(f(\sigma)\subset
Y\). The induced map
\(
T(f):T_X\to T_Y
\)
sends the piece of \(T_X\) associated to \(\sigma\) to the piece of \(T_Y\)
associated to \(f(\sigma)\). On this piece, both maps
\(
t_Y\circ T(f)\), and \(f\circ t_X\)
are given by the same linear interpolation from the image of the boundary to
the barycentre of \(f(\sigma)\). Hence they agree on every piece of \(T_X\),
and therefore agree globally. Thus
\(
t_Y\circ T(f)=f\circ t_X
\).
\end{proof}

We can now prove the required existence theorem.

\begin{theorem}\label{thm:existence}
For every integer \(n\geq 3\), there exists a finite \(n\)-dimensional
simplicial complex \(\mathcal K\) equipped with a free simplicial
\(\Ztwo\)-action such that \(\mathcal K\) is aspherical and has the integral
homology of \(\Sphere^n\).
\end{theorem}

\begin{proof}
Let \(X=\Sphere^n\) be triangulated as the boundary complex of the \((n+1)\)-dimensional
crosspolytope, and let \(a:X\rightarrow X\)
be the antipodal involution. Applying Theorem~\ref{thm:leary}, we obtain a
finite locally CAT(0) cubical complex \(T_X\) and a map
\(
t_X:T_X\rightarrow X
\).

Since \(n\geq 3\), Theorem~\ref{thm:leary} gives
\(
\dim T_X=n.
\)
Since \(t_X\) is a homology isomorphism,
\[
H_*(T_X;\Z)\cong H_*(\Sphere^n;\Z).
\]
Thus \(T_X\) has the integral homology of \(\Sphere^n\). Since \(T_X\) is
locally CAT(0), its universal cover is CAT(0), hence contractible. Therefore
\(T_X\) is aspherical.

It remains to equip \(T_X\) with a fixed-point-free involution. By the
functoriality of Leary's construction, Lemma~\ref{lem:leary-naturality}, the antipodal involution \(a:X\to X\)
induces a cubical automorphism
\(
T(a):T_X\rightarrow T_X.
\)
Since \(a^2=\mathrm{id}_X\) and \(T\) is functorial, we have
\[
T(a)^2=T(a^2)=T(\mathrm{id}_X)=\mathrm{id}_{T_X}.
\]
Thus \(T(a)\) is an involution.

We claim that \(T(a)\) is fixed-point-free. Suppose, for contradiction, that
there exists \(p\in T_X\) such that
\[
T(a)(p)=p.
\]
By the naturality of \(t_X\), we have
\[
t_X\circ T(a)=a\circ t_X.
\]
Therefore
\[
a(t_X(p))
=
t_X(T(a)(p))
=
t_X(p).
\]
Thus \(t_X(p)\) would be a fixed point of the antipodal map \(a\) on
\(\Sphere^n\), which is impossible. Hence \(T(a)\) is fixed-point-free.

Finally, take the barycentric subdivision
\(
\mathcal K:=\operatorname{sd}(T_X)
\). The cubical involution \(T(a)\) induces a simplicial involution on
\(\mathcal K\). Since \(T(a)\) is fixed-point-free on the underlying polyhedron
of \(T_X\), the induced simplicial involution on \(\mathcal K\) is also
fixed-point-free. Barycentric subdivision preserves the underlying homotopy
type and the dimension. Therefore \(\mathcal K\) is a finite
\(n\)-dimensional simplicial complex with a free simplicial \(\Ztwo\)-action,
is aspherical, and has the integral homology of \(\Sphere^n\).
\end{proof}

We now combine the previous two results.

\begin{proof}[Proof of Theorem~\ref{thm:main}]
For \(n=2\), the Simonyi, Tardos, and Vre\'cica~\cite[Section 5]{simonyi2009local} established the fact. So, let \(n\ge 3\). By Theorem~\ref{thm:existence}, there exists a finite \(n\)-dimensional free simplicial \(\Z_2\)-complex \(\mathcal{K}\) which is aspherical and satisfies
\[
\widetilde H_i(\mathcal{K};\Z)=0
\qquad\text{for all } i<n.
\]
Therefore Lemma~\ref{lem:criterion} applies and yields
\(
\coind(\mathcal{K})=1\), and \(\coind(\susp \mathcal{K})=n+1\).
This proves the theorem.
\end{proof}

\section*{Acknowledgments}
The author would like to express his sincere gratitude to Professor Ian Leary
for his generous and insightful feedback, and for helpful correspondence
concerning the functoriality and naturality properties of the metric
Kan--Thurston construction. His comments greatly improved the presentation of
the argument used in this paper.

\bibliographystyle{amsplain}
\bibliography{biblio}
\end{document}